\documentclass[final, 12pt,5p,times,twocolumn]{elsarticle}
\journal{Operations Research Letters}
\usepackage{amssymb, amsmath, amsthm}
\usepackage[dvipsnames]{xcolor}
\usepackage{hyperref}
\usepackage[ruled,norelsize]{algorithm2e}
\SetKwInOut{Hyperparameters}{Hyperparameters}
\usepackage{enumitem}
\usepackage{subcaption}
\usepackage{booktabs} 
\usepackage{siunitx}

\def\Z{\mathbb{Z}}

\def\R{\mathbb{R}}

\theoremstyle{definition}
\newtheorem{theorem}{Theorem}
\newtheorem{proposition}[theorem]{Proposition}

\newtheorem{definition}[theorem]{Definition}

\begin{document}

\begin{frontmatter}

\title{Continuous cutting plane algorithms in integer programming}

\author[poly]{Didier Ch\'etelat}
\author[cornell]{Andrea Lodi}

\affiliation[poly]{
            organization={CERC, Polytechnique Montréal},
            city={Montréal},
            country={Canada}}

\affiliation[cornell]{
            organization={Cornell Tech and Technion -- IIT},
            city={New York},
            country={USA}}

\begin{abstract}
Cutting planes for mixed-integer linear programs (MILPs) are typically computed in rounds by iteratively solving optimization problems, the so-called separation. Instead, we reframe the problem of finding good cutting planes as a continuous optimization problem over weights parametrizing families of valid inequalities. This problem can also be interpreted as optimizing a neural network to solve an optimization problem over subadditive functions, which we call the subadditive primal problem of the MILP. To do so, we propose a concrete two-step algorithm, and demonstrate empirical gains when optimizing generalized Gomory mixed-integer inequalities over various classes of MILPs. Code for reproducing the experiments can be found at \url{https://github.com/dchetelat/subadditive}.
\end{abstract}

\end{frontmatter}

\section{Introduction}

A mixed-integer linear programming (MILP) problem is an optimization problem of the form
\begin{align}
\begin{array}{rl}
    \underset{(x, z)\in\R^n}{\min} &c^tx + h^tz \\
    \text{s.t. } & Ax + Gz\geq b,\\
    &x, z\geq 0,\;\; x\in\Z^k,
\end{array}
\label{eq:milp}
\end{align}
where $A \in \R^{m \times k}$, $G \in \R^{m \times (n-k)}$, $c \in \R^k$, $h \in \R^{n-k}$ and $b \in \R^m$.
To solve such class of optimization problems, it is often useful to find lower (dual) bounds, i.e., values that optimistically approximate the MILP optimal value. Besides being useful to conservatively estimate how far from optimality a feasible solution might be, dual bounds are also useful in the branch-and-bound algorithm that forms the backbone of current state-of-the-art integer programming solvers \cite{lodi2010mixed}. The standard approach to obtain a dual bound is by solving the linear programming (LP) relaxation
\begin{align}
\begin{array}{rl}
    \underset{(x, z)\in\R^n}{\min} &c^tx + h^tz \\
    \text{s.t. } & Ax + Gz\geq b, \\
    & x, z\geq 0
\end{array}
\label{eq:lp}
\end{align}
to the original MILP \eqref{eq:milp}, i.e., the problem obtained by dropping the integrality requirement on the $x$ variables. For simplicity, we will assume both problems are feasible and bounded, so have optimal solutions.

A valid inequality for the MILP is any linear inequality $\alpha^tx +\gamma^tz \geq \beta$ that is satisfied by its feasible set, i.e., such that $\alpha^tx + \gamma^t z \geq \beta\;\; \forall (x, z)\in \{Ax + Gz\geq b, (x,z)\geq0, x\in\Z^k\}$. If furthermore there exists a point in the feasible set of the LP relaxation \eqref{eq:lp} that is not satisfied by this inequality, i.e., if there exists an $(\bar x, \bar z)\in\{Ax+Gz\geq b, (x,z)\geq 0\}$ such that $\alpha^t \bar x + \gamma^tz < \beta$, the inequality is called a cutting plane (or cut).

Algorithms exist that, given a MILP, produce cutting planes ``by construction". The resulting cutting planes are useful because the MILP with the cuts added is equivalent to the original MILP, but has a tighter linear relaxation, hence with equal or better dual bound. Usually, such algorithms proceed by solving a relaxation of the original MILP, most commonly the LP relaxation, and from its solution a collection of cutting planes, called a round, is produced. These cuts are added to the original MILP and the process is repeated again on this extended MILP, yielding a second round of cuts that can be added again to the MILP. This is repeated until convergence, or some other stopping criterion, to form a so-called cutting plane algorithm (see, e.g., \cite{lodi2010mixed}).

A major advantage of this approach is that monotone improvement is guaranteed, in that adding more cuts cannot worsen the dual bound. However, at the same time, as the number of cuts added grows, at every round the intermediate optimization problems grow in size, leading to increased computational cost, diminishing returns and reduced numerical stability.

In this paper, we propose an alternative approach to computing good cuts. Instead of devising an algorithm that produces good cuts by construction, we propose to consider a large parametrized family of inequalities that are valid by construction, but search among them for the best cuts. Our proposed algorithm alternates between solving the LP relaxation of the problem with the tentative cuts added, and taking gradient steps onto the inequality parameters so as to cut off the best solution found so far. Unlike a classical approach, the number of cuts remains fixed throughout the optimization, and the tentative cuts can be interpreted as progressively deformed or replaced in an attempt to distill an effective subset. Experiments show that this approach can find cuts leading to better dual bounds than the classical approach, although the computational cost of this search does not, for the moment, make it competitive compared with these classical methods. However, it offers a novel approach to computing cuts that completely circumvents issues regarding the increase in LP size, diminishing returns or numerical stability.

The paper is organized as follows. In Section \ref{sec:subadditive}, we introduce subadditivity as a way of generating valid inequalities for MILP \eqref{eq:milp} and we recall the interpretation of Gomory mixed-integer cuts as subadditive inequalities. In Section \ref{sec:generalized-inequalities}, we generalize the Gomory construction, while in Section \ref{sec:ccp-opt}, we formulate the optimization problem at the core of our separation algorithm. In Section \ref{sec:subadditive-opt}, we describe an interesting connection where both the classical GMI algorithm and our proposed approach can be interpreted as the training of a subadditive neural network, highlighting potential connections to machine learning. Our algorithm is presented in Section \ref{sec:algorithm} and its performance is computationally evaluated in Section \ref{sec:experiments}. Finally, we conclude in Section \ref{sec:conclusion}.

In the remainder of the paper, we will use the following notation. Inequality between two vectors $x, y\in\R^m$, denoted $x\leq y$, is taken to mean component-wise inequality, i.e., $x_i\leq y_i$ for all $1\leq i\leq m$. We also assume that functions apply to matrices column-wise: that is, for a function $f: \R^m\rightarrow\R^p$, and a matrix $X=[X_1,\dots,X_n]\in\R^{m\times n}$ with columns $X_1, \dots, X_n$, we write $f(X)$ for the $p\times n$ matrix $[f(X_1),\dots,f(X_n)]$ with columns $f(X_1), \dots, f(X_n)$. A diagonal matrix with diagonal vector $v$ will be denoted $\Delta(v)$. Let us denote $\overline{\R}=\R\cup\{-\infty,+\infty\}$ the extended real line, and $\{\cdot\}=x-\lfloor x\rfloor$ the fractional part function. The composition of two functions $f,g$ will be denoted $f\circ g$, and the concatenation of two vectors $x\in\R^{m}, x\in\R^{m'}$ will be denoted $[x, y]\in\R^{m+m'}$.

\section{Subadditive valid inequalities}
\label{sec:subadditive}

One way of generating valid inequalities is through non-decreasing, centered, subadditive functions. A function $\phi(y) : \R^m\rightarrow\R^p$ is subadditive if $\phi(y_1+y_2)\leq \phi(y_1)+\phi(y_2)$ for any $y_1, y_2$. We furthermore assume that $\phi$ is non-decreasing, that is $\phi(y_1)\leq\phi(y_2)$ for any $y_1\leq y_2$, and centered, that is $\phi(0)=0$. Let us denote $\mathcal{F}^{m,p}$ the set of all such functions.

Associated with any centered subadditive function $\phi$ is another function, its Upper Directional Derivative at Zero (UDDZ) $\bar{\phi} : \R^m\rightarrow \overline{\R}^p$, defined by
{
\setlength{\abovedisplayskip}{7pt}
\setlength{\belowdisplayskip}{7pt}
\begin{align*}
\bar{\phi}(y) = \underset{h\downarrow 0^+}{\lim\sup}\frac{\phi(hy)}{h}.
\end{align*}}%
This function is sublinear, i.e., it is subadditive and satisfies $\phi(\lambda y) = \lambda\phi(y)$ for any $y\in\R^m$ and $\lambda\geq0$. Furthermore, it upper bounds $\phi$, that is $\phi(y)\leq\bar{\phi}(y)$ for any $y \in \R^m$ \cite{guzelsoy2007duality}.

Subadditivity is closely connected to cutting planes in the following way. For a function $\phi\in\mathcal{F}^{m,p}$, the $p$ inequalities
{
\setlength{\abovedisplayskip}{7pt}
\setlength{\belowdisplayskip}{9pt}
\begin{align*}
\phi(A)x+\bar{\phi}(G)z \geq \phi(b)
\end{align*}}%
are always valid for the MILP \eqref{eq:milp} (see, e.g., \cite{guzelsoy2007duality}). That is, a subadditive function can be used as a cut-generating function.

One example is provided by Gomory mixed-integer (GMI) inequalities \cite{gom-60}. Let $B$ denote the $m\times m$ optimal basis matrix of the LP relaxation \eqref{eq:lp} in equality form,
that is, the matrix composed of the columns of $[A, G, -I]$ in the LP optimal basis. Then, the GMI inequalities are the $m$ cutting planes 
 induced by the function $\phi\in\mathcal{F}^{m,m}$ defined by
{\small
\setlength{\abovedisplayskip}{5pt}
\setlength{\belowdisplayskip}{5pt}
\begin{align*}
\phi(y) &= \min\bigg(\{B^{-1}y\},\;\Delta\,\bigg(\frac{\{B^{-1}b\}}{1-\{B^{-1}b\}}\bigg)\,\big[1-\{B^{-1}y\}\big]\bigg) \\[-3pt]
&\hspace{50pt}+ \max\bigg(-B^{-1},\;\Delta\,\bigg(\frac{\{B^{-1}b\}}{1-\{B^{-1}b\}}\bigg)\,B^{-1}\bigg)\,y,
\end{align*}}%
(see, e.g., \cite[Eq. (8)-(9)]{fischetti2007mixed}) whose UDDZ is 
{\small
\setlength{\abovedisplayskip}{5pt}
\setlength{\belowdisplayskip}{7pt}
\begin{align*}
\bar{\phi}(y) &= \max\bigg(B^{-1}y,\;-\Delta\,\bigg(\frac{\{B^{-1}b\}}{1-\{B^{-1}b\}}\bigg)\,B^{-1}y\Big)\bigg) \\[-3pt]
&\hspace{50pt}+ \max\bigg(-B^{-1},\;\Delta\,\bigg(\frac{\{B^{-1}b\}}{1-\{B^{-1}b\}}\bigg)\,B^{-1}\bigg)\,y.
\end{align*}}%
The MILP \eqref{eq:milp} with the GMI inequalities added yields the enlarged problem
{
\setlength{\abovedisplayskip}{7pt}
\setlength{\belowdisplayskip}{7pt}
\begin{align}
\begin{array}{rl}
    \underset{(x, z)\in\R^n}{\min} &c^tx + h^tz \\
    \text{s.t. } & Ax + Gz\geq b, \\
                 & \phi(A)x + \bar{\phi}(G)z \geq \phi(b) \\
                 & x, z\geq 0, \;\; x\in\Z^k,
\end{array}
\label{eq:milp-gmi}
\end{align}}%
which is equivalent to the original one \eqref{eq:milp}, but has tighter LP relaxation, hence would hopefully yield a better dual bound.  Indeed, Gomory \cite{gom-60} showed that the inequalities necessarily cut off the optimal 
LP solution i.e., they are cutting planes.

We note that if we define the function $f(y) = [y, \phi(y)]$, where we recall that $[\cdot, \cdot]$ denotes the concatenation of two vectors, then $\bar{f}(y)=[y,\bar{\phi}(y)]$, and we could have written the enlarged MILP \eqref{eq:milp-gmi} more suggestively as
{
\setlength{\abovedisplayskip}{7pt}
\setlength{\belowdisplayskip}{7pt}
\begin{align}
\begin{array}{rl}
    \underset{(x, z)\in\R^n}{\min} &c^tx + h^tz \\
    \text{s.t. } & f(A)x + \bar{f}(G)z \geq f(b), \\
                 & x, z\geq 0, \;\; x\in\Z^k.
\end{array}
\label{eq:milp-gmi2}
\end{align}}
This is a MILP in the same form as the original problem \eqref{eq:milp}, but with $A, G$ and $b$ replaced by $f(A), \bar{f}(G)$ and $f(b)$ respectively.

\section{Generalized GMI valid inequalities}
\label{sec:generalized-inequalities}

Our reasoning starts by noting that we could have generalized GMIs by instead considering a parametrized family of subadditive functions. Namely, we could have considered the functions
{
\allowdisplaybreaks
\setlength{\abovedisplayskip}{7pt}
\setlength{\belowdisplayskip}{7pt}
\begin{align}
\phi_{W, v}(y) &= \min\Big(\{Wy\},\;\Delta\Big(\frac{v}{1-v}\hspace{1pt}\Big)\,\big[1-\{Wy\}\big]\Big) 
\notag\\&\hspace{50pt}+ \max\Big(\text{-}W,\;\Delta\Big(\frac{v}{1-v}\hspace{1pt}\Big)\,W\Big)\,y
\label{eq:generalized-gmi}
\end{align}}%
parametrized by an arbitrary weight matrix $W\in\R^{m'\times m}$ and a vector $v\in{[0,1)}^{m'}$. One can recover the original function by setting $m'=m, W=B^{-1}$ and $v=\{B^{-1}b\}$. 

(We remark in passing that another family of generalized GMI inequalities has also been considered in the literature: a good summary is provided in \cite[Section 4.1]{cornuejols2008valid}. This family forms a subset of ours, namely, $\psi_W(y) = \phi_{W,\,\{Wb\}}(y)$, i.e., with $v = \{Wb\}$, and it is in fact known that, given any matrix weight $W$, that is the optimal choice for $v$ -- see, e.g., \cite{dash2010mir}, where the result is shown for the equivalent mixed-integer rounding inequalities. However, this family depends on $b$, which would have complicated our discussion as well as the resulting algorithm, so we prefer to remain in this larger family to provide a more natural derivation.)

Our generalized GMI family satisfies the following.

\begin{proposition} We have $\phi_{W, v}\in\mathcal{F}^{m,m'}$ for any $W\in\R^{m'\times m}, v\in[0,1)^{m'}$.
\end{proposition}
\begin{proof}
The function $g_v(y) = \min\big(y, \Delta(\frac{v}{1-v})[1-y]\big)$ is concave on $y\in[0, 1)^{m'}$ and non-negative. Hence, by \cite[Theorem 5.4]{matkowski2011subadditive}, the function $g_v(\{y\})$ is subadditive on $\R^{m'}$. Then, 
$g_v(\{W(y_1+y_2)\})=g_v(\{Wy_1+W_2y_2\})\leq g_v(\{Wy_1\})+g_v(\{Wy_2\})$ for any $y_1, y_2$, so $h_{W, v}=g_v(\{y\})$ is subadditive on $\R^m$. The UDDZ of $h_{W, v}$ is $\bar{h}_{W, v}=\max\big(Wy, -\Delta(\frac{v}{1-v})Wy\big)$, so that we can write
$$\phi_{W, v}(y) = h_{W, v}(y)+\bar{h}_{W, v}(-I)y.$$
Since $h_{W, v}$ is subadditive, and $y\mapsto\bar{h}_{W, v}(-I)y$ is linear, $\phi_{W, v}$ must be subadditive too.

Say $y_1\leq y_2$. Then, 
\begin{align*}
&\phi_{W, v}(y_1) = \phi_{W, v}(y_2+(y_1-y_2)) \\
&\quad\leq \phi_{W, v}(y_2) + h_{W, v}(y_1-y_2) +\bar{h}_{W, v}(-I)(y_1-y_2) \\
&\quad\leq \phi_{W, v}(y_2) + \bar{h}_{W, v}(y_1-y_2) +\bar{h}_{W, v}(-I)(y_1-y_2) \\
&\quad= \phi_{W, v}(y_2) + \bar{h}_{W, v}(-I)(y_2-y_1) +\bar{h}_{W, v}(-I)(y_1-y_2) \\
&\quad= \phi_{W, v}(y_2)
\end{align*}
because $y_2-y_1\geq 0$. Hence, $\phi_{W, v}$ must be non-decreasing. Finally, $\phi_{W, v}(0)=0$, so $\phi_{W, v}\in\mathcal{F}^{m,m'}$.
\end{proof}

In light of this proposition, functions $\phi_{W, v}(y)$ induce valid inequalities, and defining  $f_{W,v}(y) = [y, \phi_{W, v}(y)]$ analogously, we could have considered the family of MILPs
\begin{align}
\begin{array}{rl}
    \underset{(x, z)\in\R^n}{\min} &c^tx + h^tz \\
    \text{s.t. } & f_{W, v}(A)x + \bar{f}_{W, v}(G)z \geq f_{W, v}(b), \\
                 & x, z\geq 0, \;\; x\in\Z^k.
\end{array}
\label{eq:milp-enlarged}
\end{align}
parametrized by $W, v$.

Now, problem \eqref{eq:milp-enlarged} is itself a MILP. Thus, we could apply the reasoning we just outlined to itself: if we had added $m_1$ valid inequalities with weights $W_1\in\R^{m_1\times m}, v_1\in[0,1)^{m_1}$, then we could add $m_2$ new inequalities with weights $W_2\in\R^{m_2\times (m+m_1)}, v_2\in[0,1)^{m_2}$. Doing so would result in the even larger problem
\begin{align*}
\begin{array}{rl}
    \underset{(x, z)\in\R^n}{\min} &c^tx + h^tz \\
    \text{s.t. } & f_{W_2, v_2}\circ f_{W_1, v_1}(A)x \\
                 & \quad+ \bar{f}_{W_2, v_2}\circ \bar{f}_{W_1, v_1}(G)z \\
                & \hspace{80pt}\geq f_{W_2, v_2}\circ f_{W_1, v_1}(b), \\
                 & x, z\geq 0, \;\; x\in\Z^k.
\end{array}
\end{align*}
In fact, we can repeat this as many times we like, say for $K$ rounds. Then, we would end up with the MILP
\begin{align*}
\begin{array}{rl}
    \underset{(x, z)\in\R^n}{\min} &c^tx + h^tz \\
    \text{s.t. } & f_{W_K, v_K}\circ \dots \circ f_{W_1, v_1}(A)x \\
    & \quad+ \bar{f}_{W_K, v_K}\circ \dots \circ \bar{f}_{W_1, v_1}(G)z \\
    & \hspace{40pt}\geq f_{W_K, v_K}\circ \dots \circ f_{W_1, v_1}(b), \\
                 & x, z\geq 0, \;\; x\in\Z^k.
\end{array}
\end{align*}
This can be written even more suggestively by defining $\theta = (W_1, v_1, \dots, W_K, v_K)$ and
\begin{align*}
f_{\theta}(y) = f_{W_K, v_K}\circ \dots \circ f_{W_1, v_1}(y),
\end{align*}
so that the MILP resulting from adding $K$ rounds of our generalized GMI valid inequalities with weights $\theta$ would be
\begin{align}
\begin{array}{rl}
    \underset{(x, z)\in\R^n}{\min} &c^tx + h^tz \\
    \text{s.t. } & f_\theta(A)x + \bar{f}_\theta(G)z \geq f_\theta(b), \\
                 & x, z\geq 0, \;\; x\in\Z^k.
\end{array}
\label{eq:milp-family}
\end{align}
Moreover, this generalizes the classical GMI cuts, in the sense that the MILP resulting from adding $K$ rounds of GMI cuts corresponds to a specific sequence $m_1=m$, $W_1=B_1^{-1}$, $v_1=\{B_1^{-1}b_1\}$ of dimensions and weights $m_K = 2^{K-1}m$, $W_K = B_K^{-1}$, $v_K=\{B_K^{-1}b_K\}$, where $B_k$ and $b_k$ are the basis matrix and the right-hand side vector of the enlarged LP after $k$ rounds, for $1\leq k\leq K$. In other words, one could retrospectively regard the classical GMI separation algorithm as a way of computing a good parameter vector $\theta_\text{GMI}=(W_1, v_1, \dots, W_K, v_K)$ for problem \eqref{eq:milp-family}, when $m_1=m, \dots, m_K=2^{K-1}m$.

\section{Continuous cutting plane optimization}
\label{sec:ccp-opt}

One main reason why the weights $\theta_\text{GMI}$ of the classical GMI algorithm are desirable, besides their low computational cost, is that the lower bound obtained by solving the LP relaxation of problem \eqref{eq:milp-family} with dimensions $(m_1, \dots, m_K)=(m, \dots, 2^{K-1}m)$, that is
\begin{align}
\text{LP}(f_\theta) = \begin{array}{rl}
    \underset{(x, z)\in\R^n}{\min} &c^tx + h^tz \\
    \text{s.t. } & f_\theta(A)x + \bar{f}_\theta(G)z \geq f_\theta(b), \\
                 & x, z\geq 0,
\end{array}
\label{eq:lp-family}
\end{align}
with $\theta=\theta_\text{GMI}$, tends to be high (see, e.g., \cite{lodi2010mixed}). A natural question is whether those weights are optimal: that is, does $\theta=\theta_\text{GMI}$ maximize the dual bound $\text{LP}(f_\theta)$ with $(m_1,\dots,m_K)=(m, \dots, 2^{K-1}m)$? We will see empirically this is not the case, even for $K=1$.

Hence, we could instead consider the following. Let us agree ahead of time that we will have $K$ rounds of generalized GMI inequalities, with $m_1, \dots, m_K$ (not necessarily $m, \dots, 2^{K-1}m$) inequalities in each round.

\begin{definition} We call \emph{continuous cutting plane optimization} the 
unconstrained nonlinear problem
\begin{align}
\max_\theta LP(f_\theta).
\label{eq:continuous-cuts-optimization}
\end{align}
\end{definition}

That is, for a number of rounds $K$ and dimensions $m_1, \dots, m_K$ fixed \emph{a priori}, we try to find the weights that make $f_\theta$ yield the highest dual bound possible when used as a cut-generating function.

\section{Deep subadditive primal optimization}
\label{sec:subadditive-opt}

The structure of the function $f_\theta$ we are optimizing in problem \eqref{eq:lp-family} is worth a closer look. This function has a very specific structure, which reminds one of deep neural networks in machine learning. Indeed, in the language of the field, we could describe the function $f_\theta$ as a neural network with $K$ layers $f_{W_k, v_k}(y)=[y, \phi_{W_k, v_k}(y)]$, where each layer $f_{W_k, v_k} \in\mathcal{F}^{p_{k-1}, p_k}$ for any $W_k, v_k$, and where $p_k\!=\!m\!+\!m_1\!+\!\dots\!+\!m_k$. But if $f, g$ are two non-decreasing and subadditive functions, then so is $f\circ g$,  and if $f(0)=0$ and $g(0)=0$, then $f\circ g(0)=0$. Thus, $f_\theta$ must itself be in $\mathcal{F}^{m,p}$ for any $\theta$, where $p=p_K=m+m_1+\dots+m_K$ is the output dimension of $f_\theta$, and we can describe $f_\theta$ as a neural network, with a specific architecture that makes it subadditive (more precisely, in $\mathcal{F}^{m,p}$) by construction.

This is particularly enlightening when looking at the following connection. It is known \cite{jeroslow1978cutting, jeroslow1979minimal, wolsey1981integer} that solving the MILP \eqref{eq:milp} is equivalent to solving the \textit{subadditive dual} problem
\begin{flalign}
&\hspace{30pt}
\begin{array}{rl}
    \underset{g\in\mathcal{F}^{m,1}}{\max} &g(b) \\
    \text{s.t. } & g(A) \leq c, \\
                 & \bar{g}(G) \leq h,
\end{array} &
\label{eq:subadditive-dual}
\end{flalign}%
in the sense that the MILP \eqref{eq:milp} and the subadditive dual \eqref{eq:subadditive-dual} are feasible whenever the other is, and when they do they achieve the same optimal value. Let $g^*$ be an optimal solution to this latter problem. By writing $g^*={w^*}^tf^*$ for $p$-vectors $w^*=[1/p,\dots,1/p]$ and $f^*=[g^*,\dots,g^*]$, we can see that \eqref{eq:subadditive-dual} is equivalent to
\begin{flalign*}
\hspace{40pt}
&= \begin{array}{cl}
    \underset{f\in\mathcal{F}^{m,p},\,w\in\R^p}{\max} &w^tf(b) \\
    \text{s.t. } & w^tf(A) \leq c, \\
                 & w^t\bar{f}(G) \leq h, \\
                 & w\geq 0,
\end{array} &
\end{flalign*}
which then in turn equals by LP duality
\begin{flalign}
\hspace{30pt}
&= \underset{f\in\mathcal{F}^{m,p}}{\max}\left[\!\!\begin{array}{rl}
    \underset{(x,\,z)\in\R^n}{\min} &c^tx + h^tz \\
    \text{s.t. } & f(A)x + \bar{f}(G)z \geq f(b) \\
    & x, z\geq 0
\end{array}\right] & \notag\\
&= \underset{f\in\mathcal{F}^{m,p}}{\max} 
\;\text{LP}(f).
\label{eq:subadditive-primal}
\end{flalign}
By analogy with \eqref{eq:subadditive-dual}, we could call \eqref{eq:subadditive-primal} the ``subadditive primal'' problem of the MILP \eqref{eq:milp}, to which it is equivalent. Then, the continuous cuts optimization problem \eqref{eq:continuous-cuts-optimization} can be regarded as a restriction of this problem to the family $f=f_\theta\in\mathcal{F}^{m,p}$.



With this in mind, this means that the classical GMI algorithm can be interpreted as approximately solving a MILP, \emph{by training a subadditive neural network to solve the subadditive primal problem \eqref{eq:subadditive-primal} in  a greedy, layer-by-layer fashion}. In this neural network, each layer $f_{W_k, v_k}$ corresponds to a round of valid inequalities, and have the specific form
{\small\begin{align*}
f_{W_k, v_k}(y) &= \bigg[\,y,\;\;\min\Big(W_ky,\;\Delta\bigg(\frac{v_k}{1-v_k}\bigg)\,[1-\{W_ky\}]\Big)
\\&\hspace{50pt}+\max\Big(\text{-}W_k,\;\Delta\bigg(\frac{v_k}{1-v_k}\bigg)\,W_k\Big)\,y\hspace{1pt}\bigg],
\end{align*}}%
for learnable parameters $W_k, v_k$. In contrast, our continuous cuts optimization problem \eqref{eq:continuous-cuts-optimization} can be interpreted as \emph{training the same neural network, but in an end-to-end fashion} (all layers at the same time). In other words, we maintain the same generalized GMI function \eqref{eq:generalized-gmi} in our layers, but we continuously update the weights in the attempt to find better cuts.

Interestingly, this suggests that there is no strong reason why we should restrict ourselves to GMI cuts: the approach could very well be applied to any other continuously parametrized family of cuts that are induced by a subadditive function, such as Chv\'{a}tal-Gomory cuts, or even others that do not correspond to any known family. We will discuss further this possibility in the conclusion.

\section{Concrete algorithm}
\label{sec:algorithm}

Our continuous cutting plane optimization problem \eqref{eq:continuous-cuts-optimization}, although unconstrained, is highly nonlinear and standard optimization approaches quickly run into difficulties. In particular, a challenge is that most $\theta$ parameters lead to valid inequalities that do not cut off the (optimal) vertices of the LP polytope, hence lead to zero gradients.

Hence, we propose instead a two-step heuristic to solve \eqref{eq:continuous-cuts-optimization}, detailed in Algorithm \ref{alg:proposed}. The idea is to alternate between finding the optimal solution of the linear relaxation with the valid inequalities added, and changing the parameters by gradient descent so as to cut off this optimal solution. This algorithm seems to perform well in practice, and intuitively should lead to increase the dual bound $\textnormal{LP}(f_\theta)$ as more and more LP solutions are cut off. On the other hand, it is also quite computationally expensive, although warm-starting the linear program from the previous optimal basis alleviates some of that cost.

In practice, we found it useful to add some noise to the point we aim to cut off: that is, instead of trying to cut the LP solution $x^*_\text{LP}$, we try to cut a noisy version $x^*_\text{LP} + \beta\epsilon$, where $\epsilon$ is a small amount of Gaussian noise and $\beta$ is a hyperparameter (set to $10^{-4}$ in all experiments). This helped in problems where degeneracy occurred in the linear relaxations.

The parametrization of our weights $v$ require that they be contained in $[0,1)$. This could be enforced, say, with projected gradient descent, but in practice we found it easier to enforce it to be in $(0,1)$ by writing it as $v=\text{sigmoid}(v')=1/(1+\exp(-v'))$ for some other parameter $v'\in\R^m$, over which we optimize.

\begin{algorithm}[t]
\caption{Two-step heuristic}
\label{alg:proposed}
\Hyperparameters{$\alpha>0$, $\beta\geq0$}
\Repeat{$\textnormal{LP}(f_\theta)$ has converged}{
    Solve $\text{LP}(f_\theta)$ and obtain its solution $x^*_\text{LP}$
    \Repeat{$\big[f_\theta(b)-f_\theta(A)x^*_\text{LP}\big]_i < 0$ for some $i$}{
        \begin{itemize}[leftmargin=*, label=-]
        \item Compute a noisy target \\[2pt]
        \hspace{30pt}$\bar{x}_\text{LP} = x_\text{LP}^*+\beta\epsilon$ 
        \\[2pt]
        for some noise $\epsilon\sim\text{N}(0,I)$
        \vspace{-10pt}
        \item Take a gradient step 
        \vspace{-10pt}\begin{align*}\hspace{-50pt}
        \theta \;\;\text{-=}\;\; \alpha\frac1{p}\sum_{i=1}^p\nabla_\theta\Big[f_\theta(b)-f_\theta(A)\bar{x}_\text{LP}\Big]_i
        \end{align*}\vspace{-20pt}
        \end{itemize}
    }
}
\end{algorithm}

We propose two ways to initialize the weights of our inequalities. 
\begin{enumerate}
    \item We can warm-start the optimization by initializing the weights with the values obtained by running the classical GMI separation algorithm for $K$ rounds. If $m_k < 2^{k-1}m$, we can pick, for example, the weights corresponding to the $m_k$ most violated cuts.
    \item We can initialize from scratch with random weights. We attempted a few different possibilities, and our most successful approach was to initialize the rows of the $W$ matrices with $m_k$ mutually orthogonal random unit vectors, and the $v'=\text{logit}(v)=\log(v/[1-v])$ vectors from a $\text{N}(0, I_{m_k})$ distribution.
\end{enumerate}

Finally, the gradient step warrants further detail. Although in principle one could derive, through matrix calculus, closed form expressions for the gradient, we found it easier in practice to compute $f_\theta(A), f_\theta(b)$ and their gradients using an automatic differentiation library (such as Pytorch \cite{paszke2019pytorch}).

\begin{figure*}[t]
    \centering
    \begin{subfigure}{.24\textwidth}
      \includegraphics[width=\linewidth]{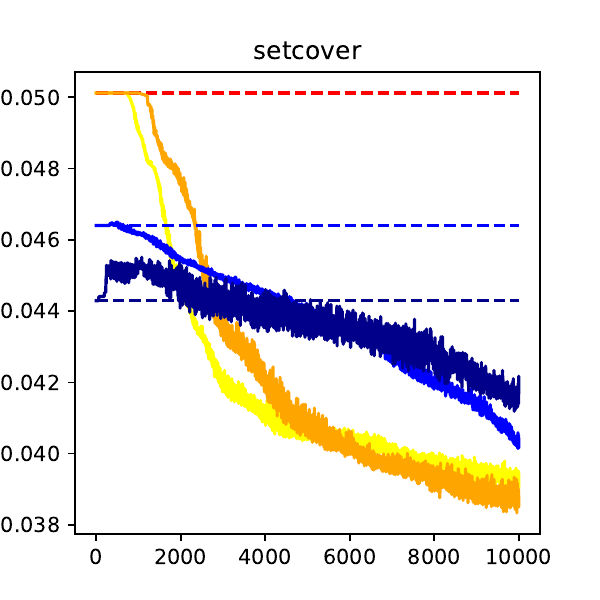}
    \end{subfigure}%
    \begin{subfigure}{.24\textwidth}
      \includegraphics[width=\linewidth]{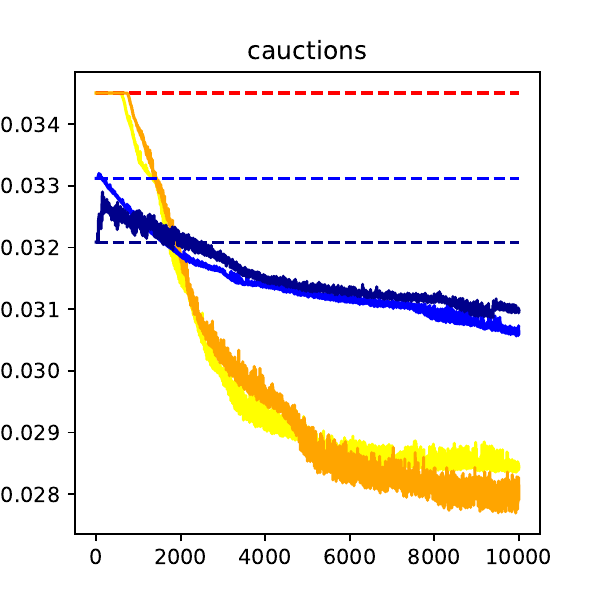}
    \end{subfigure}
    \begin{subfigure}{.24\textwidth}
      \includegraphics[width=\linewidth]{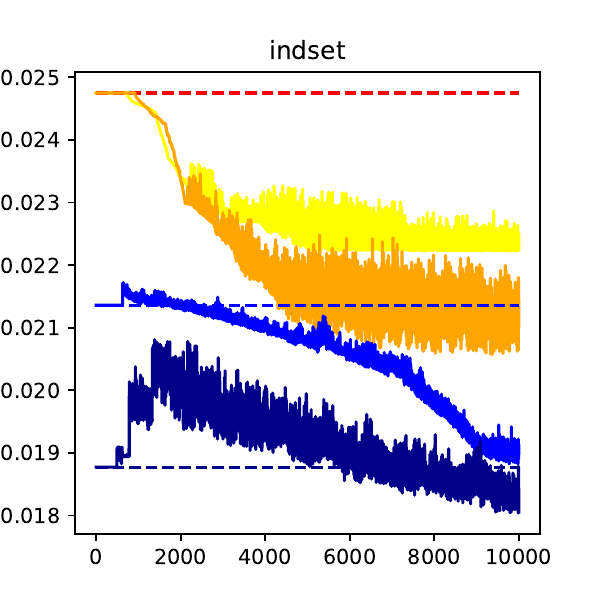}
    \end{subfigure}%
    \begin{subfigure}{.24\textwidth}
      \includegraphics[width=\linewidth]{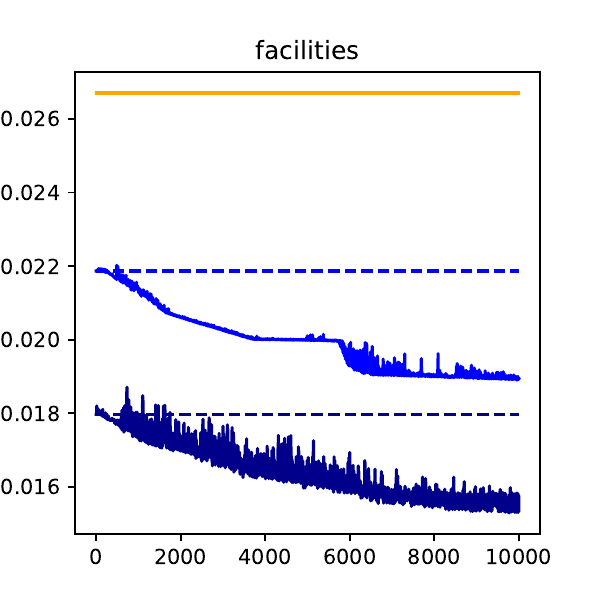}
    \end{subfigure}
        \caption{Median optimality gap over the different instances as a function of the number of gradient steps taken by Algorithm \ref{alg:proposed}. The yellow and orange lines represent the optimality gap with a random orthogonal initialization, with one and two layers, respectively. The blue and navy lines represent the optimality gap when warm-starting from the classical GMI weights, again with one and two layers respectively. The dashed lines represent the initial values in each setup, with the red dashed line being also the optimality gap of the LP relaxation.}
        \label{fig:results-1}
\end{figure*}

Now, running the classical GMI separation algorithm with an increasing number of rounds is also a way of increasing the dual bound over time. Therefore, the classical algorithm can really be seen in two ways: as a competitor for weight selection when $K$ is fixed a priori, and as a more general competitor for obtaining dual bounds when $K\rightarrow\infty$.

In the latter case, a few important differences emerge. Besides the fact we keep the number of cuts fixed (see \eqref{eq:continuous-cuts-optimization}), another difference is that improvements are not monotone anymore. That is, there is no guarantee with our algorithm that the dual bound necessarily increases as we take more gradient steps, because when varying the GMI weights so as to cut the latest LP solution, previous points that were cut might not be cut anymore. 
In particular, it is possible that a previous LP solution gets ``dislodged'' and becomes feasible again for our generalized GMI inequalities after a gradient step, worsening the dual bound. Thus, our algorithm sacrifices monotonicity, although in practice this does not seem too high a price to pay for potentially improving the dual bound.

A more abstract way we can look at the situation is that in the classical GMI separation algorithm, the memory of the previous cuts is preserved in a trivial fashion by keeping the cuts in the LP, whereas in our algorithm, the memory of the previous cuts can be thought as preserved in the slowly evolving weights. 

Note that the collection of cuts that one is forced to keep in a classical cutting plane approach depends on the sequence of fractional solutions that the algorithm needs to separate. However, we claim that such a sequence is somehow random, i.e., having very little guarantee of good convergence properties. In contrast, Algorithm \ref{alg:proposed} -- driven by the quality of the bound and giving up on monotonicity -- is able to distill a small cardinality set of cutting planes that are practically ``robust" for a cloud of fractional solutions.

\begin{figure*}[t]
    \centering
    \begin{subfigure}{.49\textwidth}
      \includegraphics[width=\linewidth]{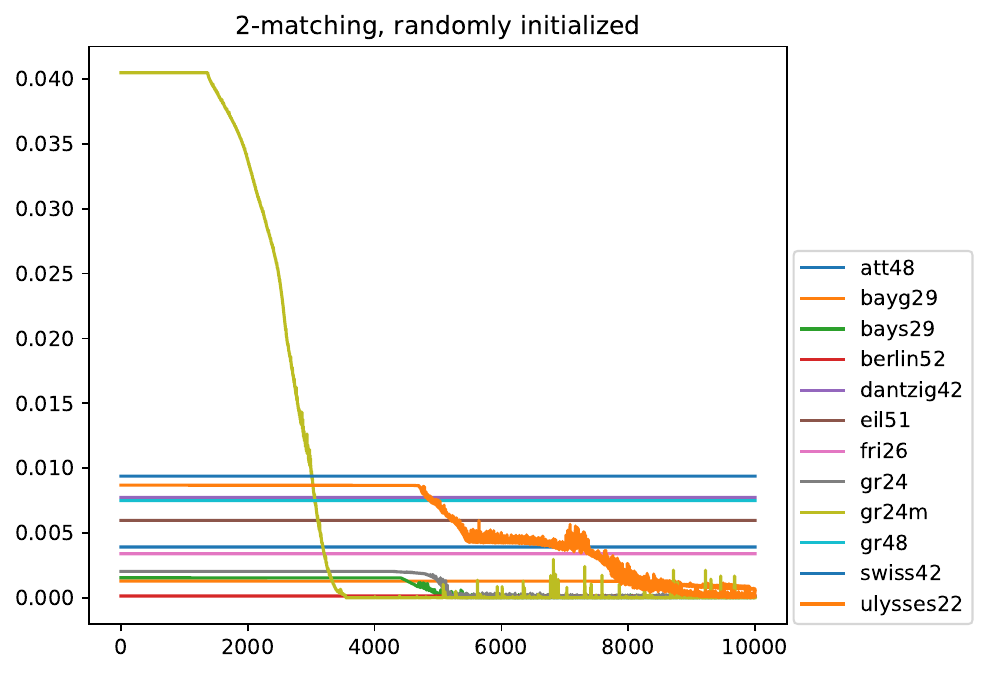}
    \end{subfigure}%
    \begin{subfigure}{.49\textwidth}
      \includegraphics[width=\linewidth]{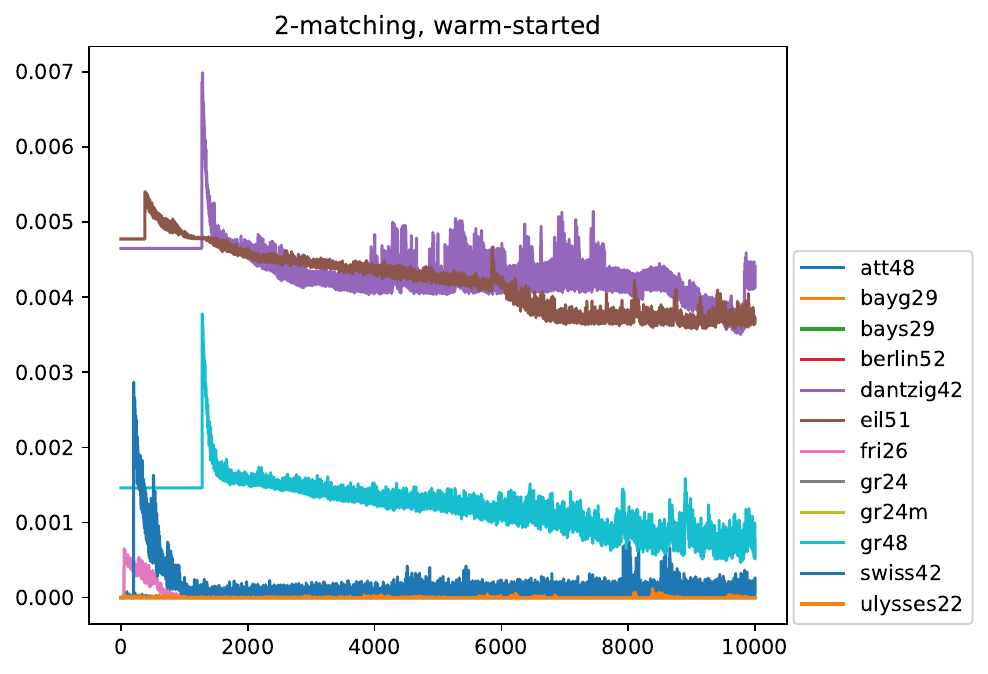}
    \end{subfigure}
    \\
    \begin{subfigure}{.49\textwidth}
      \includegraphics[width=\linewidth]{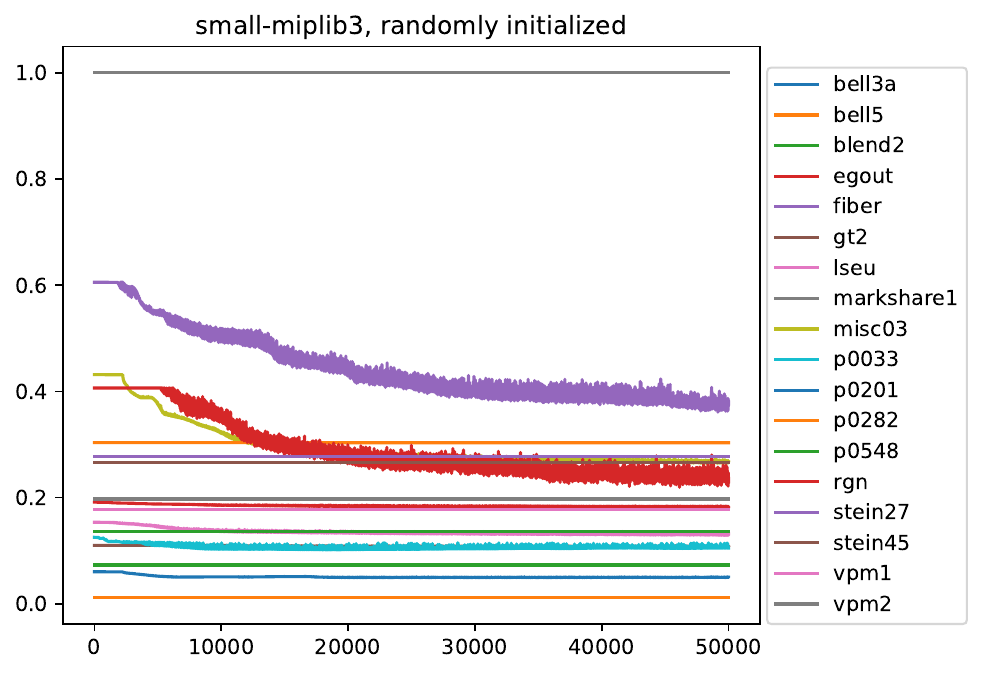}
    \end{subfigure}%
    \begin{subfigure}{.49\textwidth}
      \includegraphics[width=\linewidth]{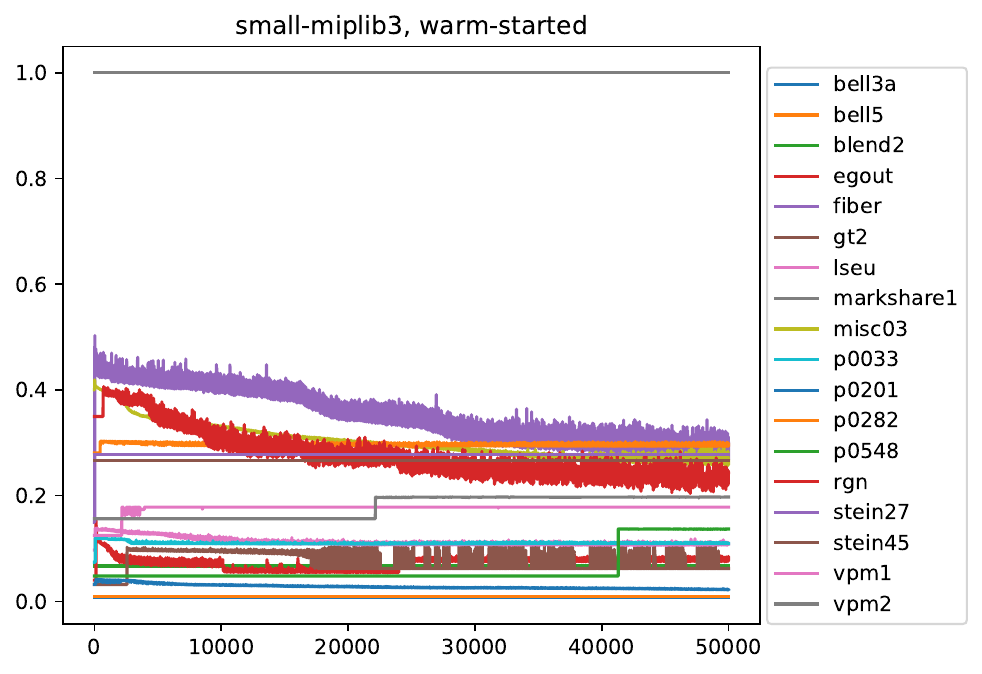}
    \end{subfigure}
        \caption{Optimality gap as a function of the number of gradient steps taken by Algorithm \ref{alg:proposed}, for each instance.}
        \label{fig:results-2}
\end{figure*}

\section{Experimental results}
\label{sec:experiments}

We now present two experiments that illustrate the performance of our approach. In each case, we implemented our algorithm with the Pytorch 1.9.0 automatic differentiation framework \cite{paszke2019pytorch}, while we used Gurobi 9.1.1 \cite{optimization2018gurobi} as LP solver, as well as for computing optimal MILP values for comparisons\footnote{\url{https://github.com/dchetelat/subadditive/commit/724e19f4fed8986f6055f3044dedfc91ea9c6cd6}.}.  The experiments were run on Intel Xeon W-2145 3.70GHz CPU cores and a NVIDIA Titan RTX GPU.

\subsection{First experiment}\label{sec:first-experiment}

For our first experiment, we evaluate on 4 benchmarks: 100 randomly generated minimum set covering, combinatorial auctions, maximum independent set and capacitated facility location instances, that were generated using the Ecole python library \cite{prouvost2020ecole} (following implementations from \cite{balas1980set,leyton2000towards,bergman2016decision,cornuejols1991comparison}, respectively.)

We report the results of optimizing a one-layer network with 32 inequalities, and a two-layer network with 32 inequalities each, illustrating the effect that deeper layers can have on performance. In addition, we show the results of starting from the classical GMI weights as well as from our random orthogonal initializer, as described in Section \ref{sec:algorithm}. Results over 10k gradient steps are reported in Figure \ref{fig:results-1}. At any gradient step, we report the median optimality gap over all instances of the same family, where optimality gap (in our formulation) is defined as $(z^*-z)/|z^*|$ for a dual bound $z$ and the MILP optimal value $z^*$. We report the computational cost of the experiments in \ref{appendix}.

As can be seen, in general our algorithm is able to learn and improve upon the classical GMI weights. This demonstrates that although Problem \eqref{eq:continuous-cuts-optimization} is highly nonconvex, it \emph{is} in fact possible to solve it to a degree where substantially better weights than the classical GMI weights can be found (albeit perhaps at great computational cost.) Moreover, using higher rank inequalities generally leads to better dual bounds, as expected.

We highlight however some less intuitive findings. First, the classical GMI weights seem to be fragile initial points, as the early gradient steps end up usually \emph{worsening} the dual bound, until the algorithm can find a better position from which actual improvements eventually occur. This is particularly dramatic with the maximum independent set instances, where for the two-layer network it took on average 6k gradient steps until an improvement over the classical weights could be found. Incidentally, this illustrates rather dramatically our point in Section \ref{sec:algorithm} that our approach does not guarantee monotone bound improvements.

Second, our random orthogonal initializer starts with the LP relaxation value, since the initial weights do not lead to inequalities that cut off any part of the LP polytope. However, interestingly, for at least two families (set covering and combinatorial auction instances), the performance eventually overtakes the curves initialized with the classical GMI weights. Thus, in at least some families, it appears that these initial inequalities might be better located for long-run optimization. At the other extreme, with capacitated facility location instances the algorithm was never able to find effective cuts when starting from the randomly initialized weights. Thus, from this experiment we cannot conclude that one initialization method is always better than the other.

\subsection{Second experiment}\label{sec:second-experiment}

We now turn to our second experiment. In this scenario, we evaluated on 12 2-matching \cite{fischetti2007optimizing} instances derived from TSPLIB \cite{reinelt1991tsplib} and 18 small instances taken from MIPLIB3 \cite{bixby1998updated}. These instances are more heterogeneous than in the first experiment, and the 2-matching problem are particularly interesting since rank-1 Chv\'atal-Gomory cuts are enough to compute their convex hull \cite{edmonds1965maximum}. Since GMI inequalities dominate Chv\'{a}tal-Gomory inequalities, this suggests that, in theory, a carefully optimized 1-layer network should be enough to solve the problems.

The instances in this experiment are harder than those in the first experiment, and we consequently use a larger setup, namely a two-layer network with 1024 inequalities in each layer. Again, we run the experiment with random orthogonal initialization, as well as initialization from the classical GMI weights. We report in Figure \ref{fig:results-2} the optimality gap achieved over each instance when running the algorithm for 10k and 50k gradient steps, on the 2-matching and MIPLIB3 instances, respectively. \ref{appendix} summarizes the computational cost of this experiment.

As can be seen, when looking at individual instances the portrait is somewhat complex. When warm-starting with the classical GMI weights, the algorithm is often able to improve on the initial value, but the curves are not always smooth, and there might be periods of stalling. With random orthogonal initialization, the algorithm often did not manage to learn an improvement within the number of gradient steps provided, although for some it did. Thus, we can conclude that although the method appears in general feasible, it offers no guarantees.

\section{Conclusions}
\label{sec:conclusion}

In this paper, we proposed an alternative way of finding cutting planes for mixed-integer linear programs, where parameters are optimized so that a subadditive neural network acts as a good cut-generating function. Besides the theoretical interest of this point of view, and its intriguing potential connection with machine learning, we also show that the problem can be solved in practice, albeit at high computational cost.

As mentioned in Section \ref{sec:ccp-opt}, we chose in this paper to illustrate our ideas with the well-known Gomory mixed-integer cuts. In practice, however, we could really consider any family of cuts induced by subadditive, non-decreasing, centered functions, such as Chv\'{a}tal-Gomory cuts, or some that might not have been previously considered in the literature. For example, we experimented with the family $\zeta_{W, v}(y) = \log(1+\phi_{W, v}(y))$. We found that on the 2-matching and small MIPLIB3 instances, these ``logarithmic'' GMI cuts often yielded superior results, which we believe could be due to yielding a loss landscape that is easier to navigate for gradient descent. Thus, there is a wide room for experimentation, and it is possible that the best families for a ``continuous'' approach are quite different than those that were best for classical approaches based on sequential LP solving.

Moreover, although this work does not represent machine learning \emph{per se} -- with machine learning being defined as the study of statistical generalization -- it is clear that the connection with deep neural networks suggests many potential ideas that could be applied to this line of work. To name a simple one, one could imagine adding a $L_1$ regularizer term in problem \eqref{eq:continuous-cuts-optimization} as to encourage the weights, hence the cuts, to be sparse. We do not doubt that many other interesting ideas could be found.

\section*{Acknowledgments}
The authors are indebted to Oktay G\"unl\"uk and many members of the Canada Excellence Research Chair in ``Data Science for Real-time Decision-making" at Polytechnique Montr\'eal where most of this work has been conducted while the second author was the chair holder. We are indebted to an anonymous referee for their very careful reading and useful suggestions. We are also grateful to the feedback and discussion during the IPAM ``Artificial Intelligence and Discrete Optimization" Workshop (February 27 - March 3, 2023).

\begin{table*}[t]
    \centering
    \resizebox{\textwidth}{!}{
    \begin{tabular}{l
                    r@{\:\( \pm \)\,}
                    r
                    r@{\:\( \pm \)\,}
                    r
                    r@{\:\( \pm \)\,}
                    r
                    r@{\:\( \pm \)\,}
                    r
                    r@{\:\( \pm \)\,}
                    r
                    r@{\:\( \pm \)\,}
                    r
                    }
         & \multicolumn{2}{c}{\textsc{setcover}}
         & \multicolumn{2}{c}{\textsc{cauctions}}
         & \multicolumn{2}{c}{\textsc{indset}}
         & \multicolumn{2}{c}{\textsc{facilities}}
         & \multicolumn{2}{c}{\textsc{2-matching}}
         & \multicolumn{2}{c}{\textsc{small-miplib3}}
        \\
        \hline
        Randomly initialized, 1 layer
        &  29.54 &   190.11
        &  22.14 &   146.06
        &  21.71 &   143.24
        &  18.39 &   102.77
        &  49.87 &    27.01
        & 354.10 &   370.82
        \rule{0pt}{8pt}
        \\
        Randomly initialized, 2 layers
        &  34.30 &   190.16
        &  26.10 &   146.78
        &  24.29 &   142.83
        &  16.92 &    87.38
        & 133.06 &   245.66
        & 534.88 &   440.83
        \\
        Warm-started, 1 layer
        &  20.36 &    50.89
        &  29.60 &   164.28
        &  11.87 &    30.18
        &  13.67 &    75.99
        & 145.97 &   240.40
        & 431.59 &   338.40
        \\
        Warm-started, 2 layers
        &  25.64 &    70.78
        &  13.13 &    13.63
        &  20.90 &    61.08
        &   5.14 &     8.23
        & 115.87 &   136.40
        & 519.75 &   357.89
        \\\hline
        Gurobi
        & 0.0040 &   0.0044
        & 0.0089 &   0.0060
        & 0.0008 &   0.0007
        & 0.0002 &   0.0001
        & 0.0001 &   0.0001
        & 0.1260 &   0.5541
        \rule{0pt}{8pt}
        \\\hline
    \end{tabular}}
    \caption{Average time in minutes ($\pm$ standard deviation) required to achieve the desired number of gradient steps in Experiments \ref{sec:first-experiment} and \ref{sec:second-experiment}. We also include the average time taken for Gurobi to solve the problems to optimality.}
    \label{tab:timings}
\end{table*}

\bibliographystyle{elsarticle-num} 
\bibliography{arxiv.bib}

\appendix

\section{Computational cost}\label{appendix}

For reference, we report in Table \ref{tab:timings} the computational cost of the experiments, as well as the times that Gurobi takes to solve these problems to optimality. As can be seen, although our method is able to find good dual bounds, the computational cost is prohibitive compared to simply feeding the instances to a standard combinatorial optimization solver. Admittedly, we have not explored any path for reducing such computational effort, and we clearly consider this an important future step.

\end{document}